\documentclass[a4paper,11pt]{article}
\usepackage{mathtools}
\mathtoolsset{showonlyrefs}

\usepackage[all,cmtip]{xy}
\usepackage{a4wide}
\usepackage{amsmath, amsthm, amssymb}
\usepackage{fontenc}
\usepackage[latin1]{inputenc}
\usepackage{graphicx}
\usepackage{pgfplots}
\usepackage[hang]{subfigure}
\usepackage{multirow}
\usepackage{array}
\usepackage{enumerate}
\usepackage{algorithmic, algorithm}

\usepackage{etoolbox}
\makeatletter
\patchcmd{\@algocf@start}{%
  \begin{lrbox}{\algocf@algobox}%
}{%
  \rule{0.1\textwidth}{\z@}%
  \begin{lrbox}{\algocf@algobox}%
  \begin{minipage}{0.8\textwidth}%
}{}{}
\patchcmd{\@algocf@finish}{%
  \end{lrbox}%
}{%
  \end{minipage}%
  \end{lrbox}%
}{}{}
\makeatother

\newtheorem{theorem}{Theorem}[section]

\newtheorem{proposition}[theorem]{Proposition}
\newtheorem{lemma}[theorem]{Lemma}
\newtheorem{remark}[theorem]{Remark}
\newtheorem{example}[theorem]{Example}
\newtheorem{corollary}[theorem]{Corollary}

\newcommand{\R}{\mathbb{R}}

\renewcommand{\atop}[2]{\genfrac{}{}{0pt}{}{#1}{#2}}

\def\tT{{\mbox{\tiny{T}}}}
\def\argmin{\mathop{\rm argmin}}

\def\lev{\mathrm{lev}\,}
\def\ker{\mathrm{ker}}
\def\supp{\mathrm{supp}}

\title{Disparity and Optical Flow Partitioning 
\\
Using Extended Potts Priors}
\author{Xiaohao Cai\thanks{University of Kaiserslautern, Department of Mathematics, Paul-Ehrlich-Str.\ 31, 67663 Kaiserslautern, Germany.
cai@mathematik.uni-kl.de, fitschen@mathematik.uni-kl.de,
and steidl@mathematik.uni-kl.de.}, \and
Jan Henrik Fitschen\footnotemark[1], \and
Mila Nikolova\thanks{CMLA, ENS Cachan, CNRS,  61 Avenue du President Wilson, F-94230 Cachan, France.
nikolova@cmla.ens-cachan.fr}, \and
Gabriele Steidl\footnotemark[1], \and
Martin Storath
\thanks{
Biomedical Imaging Group
EPFL,
CH-1015 Lausanne VD
Switzerland.
martin.storath@epfl.ch}
}
\date{}

\begin{document}

\maketitle

\begin{abstract}
This paper addresses the problems of disparity and optical flow partitioning
based on the brightness invariance assumption.
We investigate new variational approaches to these problems with Potts priors and possibly box constraints.
For the optical flow partitioning, our model includes vector-valued data and an adapted Potts regularizer.
Using the notation of asymptotically level stable functions we prove the existence of
global minimizers of our functionals.
We propose a modified alternating direction method of minimizers.
This iterative algorithm requires the computation of global minimizers of classical
univariate Potts problems which can be done efficiently by dynamic programming.
We prove that the algorithm converges both for the constrained and unconstrained problems.
Numerical examples demonstrate the very good performance of our partitioning method.
\end{abstract}

\section{Introduction}\label{sec:introduction}
An important task in computer vision is the reconstruction of three dimensional (3D)
scenes from stereo images.
Taking a photo, 3D objects are projected
onto a 2D image and the depth information gets lost.
If a stereo camera is used, two images are obtained. Due to the different perspectives there is a
displacement between corresponding points in the images
which depends on the distance of the points from the camera.
This displacement is called disparity and turns out to be inversely proportional to
the  distances of the objects, see Fig. \ref {fig:disp_ex}
for an illustration.
Therefore {\it disparity estimation} has constituted an active research area in recent years.
Global combinatorial optimization methods such as graph-cuts \cite{BVZ01,KZ01}
which rely on a discrete label space of the disparity map
and belief propagation \cite{KSK06,YWYWLN06} were developed as well as variational approaches
\cite{CEFPP12,CTKC11,DKA96,HPP12,MPP06,MPP09,WY11,WUPB12}.
In particular, in \cite{HPP12} the global energy function was also made convex by quantizing the disparity map
and converting it into a set of binary fields.
Illumination variations were additionally taken into account, e.g., in \cite{CEFPP12,CRH95}.
A stereo matching algorithm based on the curvelet decomposition was developed in \cite{MWW10}.
With the aim of reducing the computational redundancy, a histogram based
disparity estimation method was proposed in \cite{MLD11}.
Further, methods based on non-parametric local transforms
followed by normalized cross correlation (NCC) \cite{TLC03} and rank-transforms \cite{ZW94}
have been used.
In this paper we are interested in the direct {\it disparity partitioning} without a preliminary separate estimation of the disparity.
Moreover we want to avoid an initial quantization of the disparity map as necessary in graph-cut methods or
in \cite{HPP12}.
We focus on a variational approach with a linearized brightness invariance assumption
to constitute the data fidelity term. 
The Potts prior described below will serve as regularizing term  which forces the minimizer of our functional to show a
good partitioning.

%
\begin{figure}
\centering
	\includegraphics[width=0.3 \textwidth]{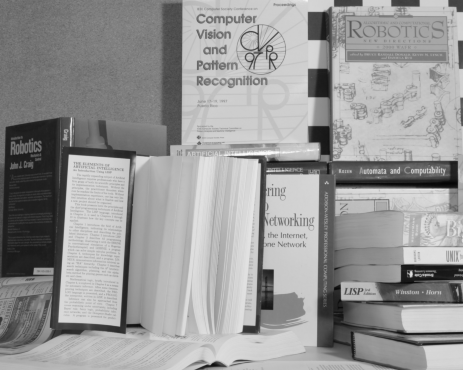}
	\hspace{0.5cm}
	\includegraphics[width=0.3 \textwidth]{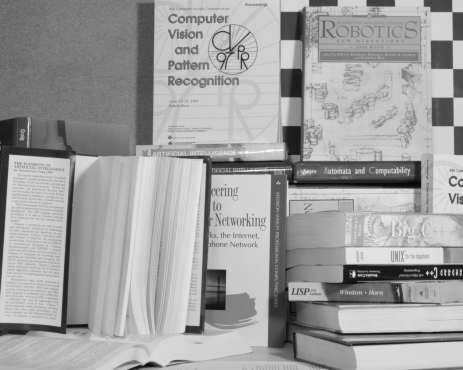}
	\hspace{0.5cm}
	\includegraphics[width=0.3 \textwidth]{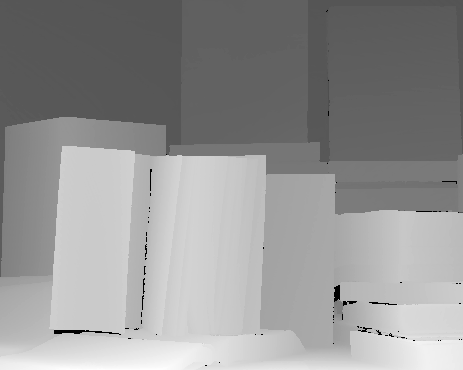}
\caption{\label{fig:disp_ex}
Left and middle:
Two images taken by a stereo camera.
The shift between the images is clearly visible.
Right:
True disparity encoded by different gray values which shows the depth of the different objects
in the scene.
(http://vision.middlebury.edu/stereo/ image credits notice)
}
\end{figure}
%

{\it Optical flow estimation} is closely related to disparity estimation where the horizontal displacement direction
has to be completed by the vertical one. In other words, we are searching for vector fields now and have to deal with vector-valued data.
Variational approaches to optical flow estimation were pioneered by Horn and Schunck \cite{HS81}
followed by a vast number of refinements and extensions,
including sophisticated data fidelity terms going beyond the brightness
\cite{BPS14,BM11,HDW13} and
nonsmooth regularizers, e.g., TV-like ones \cite{ADK99,HSSW02}
including also higher order derivatives \cite{YSM07,YSS07,YSS09a}
and nonlocal regularizers \cite{WPB10}, to mention only few of them.
In general multiscale approaches have to be taken into account to correctly determine
larger and smaller flow vectors \cite{Ana89,BBPW04,DHHM12}.
A good overview is given in \cite{BPS14}.
Recent comprehensive empirical evaluations \cite{BSLRBS11,GLSU13}
show that variational algorithms yield a very good performance.
As for the disparity we deal with variational {\it optical flow partitioning} using
the brightness invariance assumption and a vector-valued Potts prior in this paper.

The classical (discrete) Potts model, named after R. Potts \cite{Pot52}  has the form
\begin{equation} \label{potts_classic}
 \min_u \frac12 \|f - u\|_2 ^2  + \lambda \|\nabla u\|_0,
\end{equation}
where the discrete gradient consists of directional difference operators and $\| \cdot\|_0$ denotes the $\ell_0$ semi-norm.
Computing a global minimizer of the multivariate  Potts model appears to be NP hard \cite{BVZ01,DMA97,Tro06}.
For univariate data this problem can be solved efficiently using dynamic programming \cite{Cha95,FKLW08,MS85,WSD12}.
In the context of Markov random fields such kind of functionals were used by Geman and Geman \cite{GG84} and in \cite{Bes86}.
In \cite{Lec89} a deterministic continuation method to restore piecewise constant images was proposed.
A stochastic continuation approach was introduced and successfully used for the reconstruction of 3D tomographic images in \cite{RLM07}.
The method and the theory were refined in \cite{RM10}.
Recently theoretical results relating the probability for global convergence and the computation speed were given in \cite{RR13}.

There is also a rich literature on $\ell_0$-regularized methods (without additional difference operator) 
in particular in the context of sparsity and on various (convex) relaxation methods (also for data fidelity terms with linear operators).
Here we refer to the overview in \cite{FR14}. Various approximations of the $\ell_0$ norm were used in order to 
guarantee that the objective function has global minimizers; see, e.g., \cite{CJPT13},
among others. Note that the local and the global minimizers of least squares regularized with 
the $\ell_0$ norm were described in \cite{Ni13}.

In this paper, we concentrate ourselves on the (non-relaxed) Potts functional.
We apply the following model:
\begin{equation} \label{potts_gen}
 \min_{u \in S} \frac12 \|f - A u\|_2 ^2  + \lambda \|\nabla u\|_0,
\end{equation}
where $S$ is a certain compact set, $A$ a linear operator and $\|\nabla u\|_0$ a 'grouped' or vector-valued prior now.
We prove the existence of a global minimizer of the functional using the notion of asymptotically level stable functions \cite{Aus00}.
For single-valued data a completely different existence proof was given in \cite{SWD13}.
We apply an ADMM like algorithm to the general Potts model \eqref{potts_gen}.
Such algorithm was proposed for the partitioning of vector-valued images
for the Potts model \eqref{potts_classic} in \cite{SW14}.
It appears to be faster than current methods based on graph cuts and convex relaxations of the Potts model.
In particular the number of values of the sought-after image $u$ is not a priori restricted.
Our algorithm is designed for the model \eqref{potts_gen} which includes non invertible linear operators in the data fidelity term as well
as constraints.
In the context of wavelet frame operators (instead of gradients)
another minimization method for single-valued $\ell_0$-regularized, constrained problems
was suggested in \cite{Lu13,ZDL13}.
It is based on a penalty decomposition and reduces the problem mainly to the iterative solution of
$\ell_2-\ell_0$ problems via hard thresholding. Convergence to a local minimizer is shown in case of an invertible
operator $A$. However, note that in our applications both linear operators $A_1$ and $A$ have usually a nontrivial kernel.
To the best of our knowledge this is the first time that this kind of direct partitioning model was applied
for disparity and optical flow estimation.
\\

The remaining part of the paper is organized as follows:
Our disparity and optical flow partitioning models are presented in Section \ref{sec:models}.
Section \ref{sec:min} provides the proof that the (vector-valued) general Potts model
has a global minimizer.
Then, in Section \ref{sec:alg} an ADMM like algorithm is suggested together with
the convergence proofs for the constrained and unconstrained models.
Numerical experiments are shown in Section \ref{sec:experiments}.
Finally, Section \ref{sec:conclusions}
gives conclusions for future work.
%
\section{Disparity and Optical Flow Partitioning Models}\label{sec:models}
%
In this paper we deal with gray-value images $f: {\cal G} \rightarrow \R$
defined on the grid ${\cal G} := \{1,\ldots,M\} \times \{1,\ldots,N\}$ and
vector fields $u = (u_1,\ldots,u_d): {\cal G} \rightarrow \R^d$,
where $d=1$ in the disparity partitioning problem and $d=2$ in the optical flow
partitioning problem.
Note that
$$u(i,j) = (u_1(i,j),\ldots,u_d(i,j)) \in \R^d, \quad (i,j) \in  {\cal G}.$$
By $\nabla_1$, $\nabla_2$ we denote derivative operators in vertical and horizontal directions,
respectively. More precisely we will use their discrete counterparts.
Among the various possible discretizations of derivative operators we focus on forward differences
$$
\nabla_1 u (i,j) := u(i+1,j) - u(i,j), \quad \nabla_2 u (i,j) := u(i,j+1) - u(i,j)
$$
and assume mirror boundary conditions.
Further
we will need the 'grouped' $\ell_0$ semi-norm for vector-valued data
defined  by
\begin{equation} \label{group}
\|u\|_0 := \sum_{i,j=1}^n \|u(i,j)\|_0, \quad
\|u(i,j)\|_0 :=
\left\{
\begin{array}{ll}
0 &{\rm if} \; u(i,j) = 0_d,\\
1 &{\rm otherwise}.
\end{array}
\right. 
\end{equation}
Here $0_d$ denotes the null vector in $\R^d$. 
If $d=1$ then $\|u\|_0$ is the usual $\ell_0$ 'componentwise' semi-norm for vectors.
For the disparity and optical flow partitioning we will apply the $\ell_0$ semi-norm not directly to
the vectors  but rather to $\nabla_\nu u_1$ and $\nabla_\nu u$,  $\nu=1,2$, respectively, to penalize their spatial differences.
In the disparity problem we consider $\|\nabla u_1\|_0 := \|\nabla_1 u_1\|_0 + \|\nabla_2 u_1\|_0$
and the optical flow problem $\|\nabla u\|_0 := \|\nabla_1 u\|_0 + \|\nabla_2 u\|_0$.
For the later one, $\|\nabla_\nu u\|_0 = \|(\nabla_\nu u_1(i,j),\nabla_\nu u_2(i,j))_{(i,j)}\|_0$
uses indeed the 'grouped' version of the $\ell_0$ semi-norm.
\begin{remark} \label{discrete}
To have a convenient vector-matrix notation  we reorder images $f$ and $u_l$, $l=1,\ldots,d$
columnwise into vectors ${\rm vec} \, f$ and ${\rm vec} \, u_l$ of length $n := NM$.
We address the pixels by the index set $\mathbb I_n := \{1,\ldots,n\}$.
If the meaning is clear from the context we keep the notation $f$ instead of ${\rm vec} \, f$ .
In particular we will have $u_l \in \R^n$ and $u = (u_1^\tT,\ldots,u_d^\tT)^\tT \in \mathbb R^{nd}$.
After columnwise reordering the forward difference operators (with mirror boundary conditions) can be written as matrices
\begin{equation} \label{nabla_discr}
\nabla_1 := I_d \otimes I_M \otimes D_N, \quad \nabla_2 := I_d \otimes D_M^\tT \otimes I_N,
\end{equation}
where $I_N$ denotes the $N\times N$ identity matrix,
$$
D_N :=
\left(
\begin{array}{cccccc}
-1 & 1\\
  & -1 & 1\\
  &   &   & \ddots\\
  &   &   &   & -1 & 1\\
  &   &   &   &   & 0
\end{array}
\right) \in \mathbb R^{N,N}
$$
and $\otimes$ is the tensor (Kronecker) product of matrices.
\end{remark}
Using the indicator function of a set $S$ defined by
\[
\iota_{S} (t) =
\begin{cases}
0 &{\rm if} \ t \in S, \\
\infty &{\rm otherwise},
\end{cases}
\]
we can address box constraints on $u$ by adding the regularizing term $\iota_{S_{Box}}(u)$,
where $$S_{Box} := \{u \in \mathbb R^{dn}: u_{min} \le u \le u_{max}\}.$$

Both in the disparity and optical flow partitioning problems we are given a sequence of images.
In this paper we focus on two images $f_1$ and $f_2$ coming
from
(i) the appropriate left and right images taken, e.g., by a stereo camera (disparity problem),
and
(ii) two image frames at different times arising, e.g., from a video (optical flow problem).
Then the models rely on an invariance requirement between these images.
Various invariance assumptions were considered in the literature
and we refer to \cite{BPS14} for a comprehensive overview.
Here we focus on the brightness invariance assumption.
In the disparity model we address only horizontal displacements and consider in a continuous setting
\begin{equation} \label{bas_disp}
f_1(x,y) - f_2(x - u_1(x,y),y) \approx 0.
\end{equation}
For the optical flow model we assume
\begin{equation} \label{bas_flow}
f_1(x,y) - f_2\big((x,y) - u(x,y) \big) \approx 0, \quad u := (u_1,u_2).
\end{equation}
Using first order Taylor expansions around an initial disparity $\bar u_1$, resp., 
an initial optical flow estimate $\bar u = (\bar u_1,\bar u_2)$,
gives
\begin{align}
{\rm disp.}&: \; f_2(x- u_1,y) \approx f_2(x- \bar u_1,y) -
\nabla_1 f_2(x - \bar u_1,y) (u_1 (x,y) - \bar u_1(x,y)),
\\
{\rm flow}&: \; f_2\big( (x,y) - u) \approx f_2 \big((x,y) - \bar u)\big) -
(\nabla_1 f_2((x,y) - \bar u),  \nabla_2 f_2((x,y) - \bar u) \big) (u(x,y) - \bar u(x,y)) .
\end{align}
To get an initial disparity we will use a simple block-matching approach with NCC as measure for the block similarity,
following the ideas in \cite{CEFPP12,TLC03}.
Then the linearized invariance requirements \eqref{bas_flow} and \eqref{bas_disp} become
\begin{align}
{\rm disp.}&: \; 0  \approx f_1(x,y) - f_2(x- \bar u_1,y) + \nabla_1 f_2(x - \bar u_1,y) (u_1(x,y) - \bar u_1(x,y)),
\\
{\rm flow}&: \; 0  \approx f_1(x,y) - f_2\big( (x,y) - \bar u \big) +
\big(\nabla_1 f_2((x,y) - \bar u),  \nabla_2 f_2((x,y) - \bar u) \big) (u(x,y) - \bar u(x,y)) .
\end{align}
Note that $f_2((x,y) - \bar u)$ is only well defined in the discrete setting if $(i,j) - \bar u$ is in $\cal{G}$.
Later we will see that our method to compute $\bar u$ really fulfills this condition,
thus we can carry over the continuous model to the discrete setting without any modifications.
Using a non-negative increasing function
$\varphi:\mathbb R_{\ge 0} \rightarrow \mathbb R$, and considering only grid points
$(x,y) = (i,j) \in {\cal G}$
the data term for the disparity partitioning model becomes for example
$$
\sum_{(i,j) \in {\cal G}} \varphi \big(\nabla_1 f_2(i - \bar u_1,j) u_1(i,j) -
(\nabla_1 f_2(i - \bar u_1,j) \bar u_1(i,j) + f_2(i- \bar u_1,j) - f_1(i,j) ) \big).
$$
In this paper we will deal with quadratic functions $\varphi(t) := \frac12 t^2$.
Using the notation in Remark \ref{discrete} our partitioning models become 
\begin{align} \label{e_disp}
{\rm disp.}: \; E_{\rm disp} (u_1)  & := \frac12\| A_1 u_1 - b_1\|_2^2 + \mu \, \iota_{S_{Box}}(u_1) +
\lambda \left(\|\nabla_1 u_1 \|_0 + \|\nabla_2 u_1 \|_0 \right),
\\
{\rm flow}: \; E_{\rm flow} (u)   &:= \frac12\| A u - b\|_2^2 + \mu \, \iota_{S_{Box}}(u) +
\lambda \left(\|\nabla_1 u \|_0 + \|\nabla_2 u \|_0 \right), \label{e_flow}
\end{align}
where $\mu \in \{0,1\}$, $\lambda > 0$, 
$\|\cdot\|_0$ stands for the 'group' semi-norm in \eqref{group} and
\begin{align} \label{a1}
A_1 &:= {\rm diag} \left({\rm vec} \big( \nabla_1 f_2(i-\bar u_1,j) \big)\right), \\
A &:= \left({\rm diag} \left({\rm vec} \big( \nabla_1 f_2((i,j) - \bar u) \big)\right),  {\rm diag} \left({\rm vec} \big( \nabla_2 f_2((i,j) - \bar u) \big)\right) \right),
\label{only_a}\\
b_1 &:= {\rm vec} \big(\nabla_1 f_2(i - \bar u_1,j) \bar u_1(i,j)  + f_2(i- \bar u_1,j) - f_1(i,j) \big),
\label{b_1}\\
b &:= {\rm vec} \left( \big(\nabla_1 f_2((i,j) - \bar u),  \nabla_2 f_2((i,j) - \bar u) \big) \bar u(i,j)
+ f_2 ( (i,j) - \bar u) - f_1(i,j) \right).
\label{only_b}
\end{align}
We are looking for minimizers of these functionals.
%
\section{Global Minimizers for Potts Regularized Functionals}\label{sec:min}
%
We want to know if the functionals in \eqref{e_disp} and \eqref{e_flow} have global minimizers. 
Both $E_{\rm disp}$ and $E_{\rm flow}$ are lower semi-continuous (l.s.c.) and proper functionals.
When $\mu=1$, the minimization of $E_{\rm disp}$ and $E_{\rm flow}$ is constrained to
the {\em compact} set $S_{Box}$ in which case
\eqref{e_disp} and \eqref{e_flow} have global minimizers; see, e.g.,
\cite[Proposition 3.1.1,~p. 82]{AT03}.

Next we focus on the case  $\mu = 0$.
More general, we consider for arbitrary given $A \in \R^{n,dn}$, $b \in \R^n$ and $p\ge 1$
functionals $E: \mathbb R^{dn} \rightarrow \R$ of the form
\begin{equation} \label{fd}
E(u) := \frac1p\| A u - b\|_p^p  +
\lambda \left(\|\nabla_1 u \|_0 + \|\nabla_2 u \|_0 \right), \qquad \lambda > 0.
\end{equation}
The existence of a global minimizer was proved in the case $d=1$ in \cite{SWD13}.
Here we give a shorter and more general proof that holds for any $d \ge 1$ using the notion of
asymptotically level stable functions.
This wide class of functions was introduced by Auslender~\cite{Aus00} in 2000 and since then 
it appeared 
that many problems on the existence of optimal solutions 
are easily solved for these functions.
As usual,
\[\lev(E,\lambda) := \{u\in \R^{dn}: E(u)\leq \lambda \} \quad  {\rm for} \quad \lambda >\inf_u E(u)~;\]
by  $E_\infty$ we denote the {\it asymptotic (or recession) function} of $E$ and
$$
\ker(E_\infty) := \{ u\in\R^{dn}: E_\infty(u) = 0\}
.$$
The following definition is taken from \cite[p.~94]{AT03}:
a l.s.c.  and proper function $E:\R^{dn} \rightarrow \R \cup \{+\infty\}$
is said to be {\sl asymptotically level stable (als)}
if for each $\rho>0$,
each real-valued, bounded sequence $\{\lambda_k\}_k$ and
each sequence $\{u_k\}\in\R^{dn}$ satisfying
\begin{equation} \label{aal}
u_k \in \lev(E,\lambda_k), \quad \|u_k\| \rightarrow +\infty, \quad \frac{u_k}{\|u_k\|} \rightarrow \tilde u \in \ker (E_\infty),
\end{equation}
there exists $k_0$ such that
\begin{equation} \label{lsa}
u_k- \rho \tilde u \in \lev(E,\lambda_k)\quad \forall k\geq k_0.
\end{equation}
If for each real-valued, bounded sequence $\{\lambda_k\}_k$ there exists no sequence $\{u_k\}_k$
satisfying \eqref{aal}, then $E$ is automatically als. 

In particular, coercive functions are  als.
It was originally exhibited in \cite{BBGT98} (without the
notion of als functions) that any als function $E$
with $\inf E>-\infty$ has a global minimizer.
The proof is also given in \cite[Corollary 3.4.2]{AT03}.
We show that the discontinuous non-coercive objective $E$ in \eqref{fd} 
is als and has thus a global minimizer.
%
\begin{theorem} \label{pal}
Let $E:\R^{dn} \rightarrow \R$ be of the form \eqref{fd}.
Then the following relations hold true:
\begin{itemize}
\item[i)] $\ker (E_\infty )=\ker(A)$.
\item[ii)] $E$ is als.
\item[iii)] $E$ has a global minimizer.
\end{itemize}
\end{theorem}
%
\begin{proof}
i) The asymptotic function $E_\infty$ of $E$ can be calculated
according to \cite{Ded77}, see also \cite[Theorem 2.5.1]{AT03}, as
\[
E_\infty (u) =
\liminf_{\atop{u'\rightarrow u}{t\rightarrow\infty}}\frac{E(tu')}{t}.
\]
Then
\begin{align}
E_\infty(u)
&=\liminf_{\atop{u'\rightarrow u}{t\rightarrow\infty}} \frac{ \frac1p\|At u'-b\|_p^p + \|\nabla_1 (t u')\|_0 + \|\nabla_2 (t u')\|_0}{t} \\
&=\liminf_{\atop{u'\rightarrow u}{t\rightarrow\infty}} \left( \frac1p t^{p-1} \|A u'- \frac1t b\|_p^p + \frac{\|\nabla_1 (t u')\|_0 + \|\nabla_2 (t u')\|_0}{t} \right) \\
&=\left\{
\begin{array}{lll}
0 & {\rm if} & u\in \ker(A),\\
+\infty & {\rm if} & u \not\in \ker(A) \; {\rm and} \; p > 1,\\
\|A u\|_1 & {\rm if} & u \not\in \ker(A) \; {\rm and} \; p = 1,
\end{array}
\right.
\label{auv}
\end{align}
and consequently
$\ker (E_\infty) = \ker(A)$.

ii)
Let $\{u_k\}_k$ satisfy \eqref{aal} with $u_k\,\|u_k\|^{-1}  \rightarrow \tilde u \in \ker(A)$
and let $\rho>0$ be arbitrarily fixed.
Below we compare the numbers
$\|\nabla_\nu u_k\|_0$ and $\|\nabla_\nu (u_k-\rho \tilde u)\|_0$, $\nu = 1,2$.
There are two options.
\\
If $(i,j) \in \supp(\nabla_1 \tilde u) := \{(i,j) \in {\cal G}: \tilde u(i+1,j) - \tilde u(i,j) \not = 0_d\}$,
then
\[
\tilde u(i,j) - \tilde u(i+1,j) = \lim_{k\rightarrow\infty} \frac{u_k (i,j) - u_k (i+1,j)}{\|u_k\|} \neq 0_d
\]
and
$\| u_k(i,j) - u_k(i+1,j) \|>0$
for all but finitely many $k$.
Therefore, there exists $k_1(i,j)$ such that
\begin{equation} \label{wb}
\| u_k(i,j) - u_k (i+1,j) - \rho (\tilde u(i,j) -  \tilde u(i+1,j) ) \|_0 \leq \| u_k(i,j) - u_k(i,j+1) \|_0
\quad
\forall k \geq k_1(i,j).
\end{equation}
If $(i,j) \in {\cal G} \backslash \supp(\nabla_1 \tilde u)$, i.e., $\tilde u(i,j) - \tilde u(i+1,j) = 0_d$, then clearly
\begin{equation} \label{wc}
u_k(i,j) - u_k(i+1,j) - \rho (\tilde u(i,j) - \tilde u(i+1,j) ) = u_k(i,j) - u_k(i+1,j) .
\end{equation}
Combining \eqref{wb} and \eqref{wc} shows that
\begin{equation}\label{we}
\| u_k(i,j) - u_k(i+1,j) - \rho (\tilde u(i,j) - \tilde u(i+1,j) ) \|_0
\leq
\| u_k(i,j) - u_k(i+1,j) \|_0 \quad \forall k\geq k_1(i,j)
\end{equation}
and hence
\begin{equation}\label{we1}
\| \nabla_1(u_k - \rho \tilde u) \|_0 \leq \| \nabla_1 \, u_k \|_0
\quad \forall k \geq k_1 := \max \{k_1(i,j): (i,j) \in {\cal G}\}.
\end{equation}
In the same way, there is $k_2$
so that
\begin{equation} \label{wea}
\|\nabla_2 (u_k - \rho \tilde u)\|_0 \leq \|\nabla_2 u_k\|_0 \quad \forall k \geq k_2.
\end{equation}
By part i) of the proof we know that $A \tilde u = 0_n$ which jointly
with~\eqref{we1} and \eqref{wea}
implies
for all $k \ge k_0:=\max\{k_1,k_2\}$ that
\begin{align}
E(u_k- \rho \tilde u)
& = \frac1p \|A(u_k - \rho\tilde u ) - b\|_p^p + \lambda (\|\nabla_1 (u_k - \rho \tilde u)\|_0 + \|\nabla_2 (u_k - \rho \tilde u)\|_0)\\
& = \frac1p \|A u_k-b\|_p^p + \lambda (\|\nabla_1 (u_k - \rho \tilde u) \|_0 + \|\nabla_2 (u_k - \rho \tilde u)\|_0)\\
& \leq \frac1p \|Au_k-b\|_p^p + \lambda (\|\nabla_1 u_k\|_0 + \|\nabla_2 (u_k) \|_0) = E(u_k) .
\label{hb}
\end{align}
Hence it follows  by $u_k \in \lev(E,\lambda_k)$ that $u_k - \rho \tilde u \in \lev(E,\lambda_k)$ for any $k\geq k_0$.
Consequently $E$ is als.
\\
Finally, iii) follows directly from \cite[Corollary 3.4.2]{AT03}.
\end{proof}
%
\section{ADMM-like Algorithm}\label{sec:alg}
In this section we follow an idea in \cite{SW14} to approximate
minimizers of our more general functionals $E_{\rm disp}$ and $E_{\rm flow}$.
Basically the problem is reduced to the iterative computation of minimizers of the univariate classical Potts problem
for which there exist efficient solution techniques using dynamic programming \cite{FKLW08}.
Here we apply the method proposed in \cite{WSD12,SWD13}.
We consider
\begin{equation}\label{model_general}
		\min_{u \in \mathbb R^{nd}}  \Big\{  F(u) + \lambda \big( \|\nabla_1 u\|_0 + \|\nabla_2 u\|_0 \big)\Big\}	.
\end{equation}
Clearly, we have
\begin{align} \label{data_disp}
{\rm disp.} \; (d = 1): \quad F(u) &:= \frac12 \|A_1 u - b_1\|^2_2 + \mu \, \iota_{S_{Box}}(u), \quad u = u_1,\\
{\rm flow} \; (d = 2): \quad F(u) &:= \frac12 \|A u - b\|^2_2 + \mu \, \iota_{S_{Box}}(u), \quad u = (u_1^\tT,u_2^\tT)^\tT.\label{data_flow}
\end{align}
For $\mu = 1$ we have a (box) constrained problem; for $\mu = 0$ an unconstrained one.
In \cite{SW14} partitioning problems of vector-valued images with
$F(u) := \frac12 \|u - b\|_2^2$ were considered. 
In our setting a linear operator is involved into the data
term which is not a diagonal operator in the optical flow problem, see \eqref{a1}, 
and in both cases \eqref{a1} and \eqref{only_a} it has a non-trivial kernel.
Further, we may have box constraints in addition.
The minimization problem can be rewritten as
\begin{equation}
		\min_{u,v,w \in \mathbb R^{nd}}  \Big\{  F(u) + \lambda \big( \|\nabla_1 v\|_0 + \|\nabla_2 w\|_0 \big) \quad \mbox{subject to}
		\quad v = u,\; w=u\Big\}	.
\end{equation}
To find an approximate (local) minimizer we suggest the following algorithm which
resembles the basic structure of an alternating direction method of multipliers (ADMM) \cite{BPCPE10,Gab83} but
with inner parameters $\eta^{(k)}$ which has to go to infinity.
%
\begin{algorithm}[H]
\caption{ADMM-like Algorithm \label{A1}}
\begin{algorithmic}
\STATE \textbf{Initialization:}  $v^{(0)},w^{(0)},q_1^{(0)},q_2^{(0)},\eta^{(0)}$ and $\sigma > 1$
\STATE \textbf{Iteration:} For $k = 0,1,\ldots$ iterate
\begin{align} \label{admm_sol_u_gen}
u^{(k+1)} &\in  \argmin_{u} \Big\{
	F(u) + \frac{\eta^{(k)}}{2} \big(\|u-v^{(k)}+q_1^{(k)}\|_2^2
		+ \|u-w^{(k)}+q_2^{(k)}\|_2^2
	\big) \Big\},
\\
\label{admm_sol_x}
v^{(k+1)} & \in \argmin_v \Big\{
	\lambda \|\nabla_1 v\|_0 + \frac{\eta^{(k)}}{2} \|u^{(k+1)}-v+q_1^{(k)}\|_2^2  \Big\},
\\
\label{admm_sol_y}
w^{(k+1)} & \in	\argmin_w \Big\{
	\lambda \|\nabla_2 w\|_0 + \frac{\eta^{(k)}}{2} \|u^{(k+1)}-w+q_2^{(k)}\|_2^2  \Big\},
\\
\label{b1}
q_1^{(k+1)} &= q_1^{(k)} + u^{(k+1)} - v^{(k+1)},
\\
\label{b2}
q_2^{(k+1)} &= q_2^{(k)} + u^{(k+1)} - w^{(k+1)},
\\
\label{eta}
\eta^{(k+1)} &= \eta^{(k)} \sigma.
\end{align}
\end{algorithmic}
\end{algorithm}
%
Step 1 of the algorithm in \eqref{admm_sol_u_gen} can be computed
for our optical flow term $F$ in \eqref{data_flow}
and $\mu = 0$ by setting the gradient of the respective function to zero.
Then  $u^{(k+1)}$
is the solution of the linear system of equations
$$
(A^\tT A + 2 \eta^{(k)} I_{dn}) u = A^\tT b + \eta^{(k)} \left( v^{(k)}-q_1^{(k)} + w^{(k)}-q_2^{(k)}\right).
$$
For the disparity problem \eqref{data_disp} we have just to replace $A$ by $A_1$ which is a simple diagonal matrix
and $b$ by $b_1$.
For $\mu = 1$ and the disparity problem, $u^{(k+1)}$ can be computed componentwise by straightforward computation as
$$
u^{(k+1)} = \max\left\{\min\{ u^{(k+\frac12)}, u_{\max} \}, u_{\min} \right\},
$$
where
\begin{equation} \label{lin_syst}
u^{(k+\frac12)} := (A_1^\tT A_1 + 2 \eta^{(k)} I_{n})^{-1}
\left(A_1^\tT b_1 + \eta^{(k)} \left( v^{(k)}-q_1^{(k)} + w^{(k)}-q_2^{(k)} \right) \right).
\end{equation}
For the optical flow problem and  $\mu = 1$ we have to minimize
a box constrained quadratic problem for which there exist efficient algorithms, see, e.g., \cite{BNO03}.
In our numerical part the optical flow problem is handled without constraints, i.e. for $\mu = 0$.
In this case, only the linear system of equations \eqref{lin_syst} has to be solved.

The Steps 2 and 3 in \eqref{admm_sol_x} and \eqref{admm_sol_y} 
are univariate Potts problems which can be solved efficiently using the method proposed in \cite{SW14,WSD12}.
As shown in \cite{SW14} the vector-valued univariate Potts problem can be tackled nearly in the same way as in the scalar-valued case.
The arithmetic complexity is ${\cal O}(dn^\frac32)$ if $N \sim M$.
\\

Next we prove the convergence of Algorithm \ref{A1}.
Due to the NP hardness of the problem we can in general not expect that the limit point
is in general a (global) minimizer of the cost function.
First we deal with a general situation which involves our unconstrained problems ($\mu = 0$).
We assume that any vector in the subdifferential $\partial F$ of $F$ fulfills the growth constraint
\begin{equation} \label{subdiff_prop}
	u^* \in \partial F(u) \quad \Rightarrow \quad \|u^*\|_2 \le C(\|u\|_2 + 1).
\end{equation}
It can be easily checked that $F: \mathbb R^{dn} \rightarrow \mathbb R^n$ with
$F(u) := \frac1p\|M u - m\|_p^p$, $p \in [1,2]$ fulfills \eqref{subdiff_prop}
for any matrix $M \in \mathbb R^{n,dn}$ and $m \in \R^n$.
Note that the variable $C$ stands for any constant in the rest of the paper.
%
\begin{theorem} \label{thm1}
		Let $F:\mathbb R^{dn} \rightarrow \mathbb R \cup \{ + \infty \}$ be a proper, closed, convex function which fulfills \eqref{subdiff_prop}.
		Then Algorithm \ref{A1} converges in the sense that
		$(u^{(k)}, v^{(k)}, w^{(k)}) \rightarrow (\hat u, \hat v, \hat w)$ as $k \rightarrow \infty$
                with $\hat u = \hat v = \hat w$ and $(q_1^{(k)}, q_2^{(k)}) \rightarrow (0,0)$ as $k \rightarrow \infty$.
\end{theorem}
	\begin{proof}
		By \eqref{b1} we have
		\begin{align}\label{eqn-min-g1}
			\frac{\eta^{(k)}}{2}\|q_1^{(k+1)} \|_2^2  
			&= 
		  \frac{\eta^{(k)}}{2}\|u^{(k+1)}-v^{(k+1)}+q_1^{(k)}\|_2^2 \\
		  &\le   
			\lambda \|\nabla_1 v^{(k+1)}\|_0
			+ \frac{\eta^{(k)}}{2}\|u^{(k+1)}-v^{(k+1)}+q_1^{(k)}\|_2^2 
			\end{align}
and by \eqref{admm_sol_x} further
\begin{align}	  
\frac{\eta^{(k)}}{2}\|q_1^{(k+1)} \|_2^2  
&\le  \lambda \|\nabla_1(u^{(k+1)}+q_1^{(k)})\|_0 + \frac{\eta^{(k)}}{2}\|u^{(k+1)}- (u^{(k+1)}+q_1^{(k)} )+q_1^{(k)}\|_2^2 \\
&\le  \lambda \|\nabla_1(u^{(k+1)}+q_1^{(k)})\|_0 \\
&\le \lambda n.
\end{align}
By \eqref{b2} and \eqref{admm_sol_y} we conclude similarly
		\begin{equation} \label{eqn-min-h1}
		\frac{\eta^{(k)}}{2}\|q_2^{(k+1)} \|_2^2  \le \lambda n.
		\end{equation}
		Hence it follows
		\begin{equation}\label{eqn-min-b2-b3}
		\|q_1^{(k+1)} \|_2^2  \le  \frac{2\lambda n}{\eta^{(k)}}
		\quad {\rm and}\quad
		\|q_2^{(k+1)} \|_2^2  \le  \frac{2\lambda n}{\eta^{(k)}},
		\end{equation}
		which implies
		$q_1^{(k+1)} \rightarrow 0$ and $q_2^{(k+1)} \rightarrow 0$ as $k\rightarrow \infty$.
		Further, we obtain by $u^{(k)} - v^{(k)} = q_1^{(k)} - q_1^{(k-1)}$ that
		\begin{equation} \label{star2}
		\|v^{(k)} - u^{(k)} \|_2  	\le  \| q_1^{(k)} \|_2 + \| q_1^{(k-1)} \|_2
					\le  \sqrt{\frac{2\lambda n}{\eta^{(k-1)}}} + \sqrt{\frac{2\lambda n}{\eta^{(k-2)}}}
					\le  2\sqrt{\frac{2\lambda n}{\eta^{(k-2)}}}
		\end{equation}
		and analogously
		\begin{equation} \label{star3}
		\|w^{(k)} - u^{(k)} \|_2  	\le  2\sqrt{\frac{2\lambda n}{\eta^{(k-2)}}}.
		\end{equation}
		For
$\epsilon(k) := v^{(k)}-u^{(k)}-q_1^{(k)} + w^{(k)}-u^{(k)}-q_2^{(k)}$
		we get by \eqref{eqn-min-b2-b3} - \eqref{star3} that
		\begin{align} \nonumber
		  \| \epsilon(k) \|_2	&\le \| q_1^{(k)} \|_2 + \| q_2^{(k)} \|_2 + \| v^{(k)} - u^{(k)} \|_2 + \| w^{(k)} - u^{(k)} \|_2 \\
					&\le \sqrt{\frac{2\lambda n}{\eta^{(k-1)}}} + \sqrt{\frac{2\lambda n}{\eta^{(k-1)}}}
                                        + 2\sqrt{\frac{2\lambda n}{\eta^{(k-2)}}} + 2\sqrt{\frac{2\lambda n}{\eta^{(k-2)}}}
					\le 6\sqrt{\frac{2\lambda n}{\eta^{(k-2)}}}, \label{important}
		\end{align}
i.e., $\| \epsilon(k)\|_2$ decreases exponentially.
		By Fermat's theorem the proximum $u^{(k+1)}$ in \eqref{admm_sol_u_gen} has to fulfill
$$0 \in \partial F(u^{(k+1)} ) + \eta^{(k)} (u^{(k+1)} - v^{(k)} + q_1^{(k)} + u^{(k+1)} - w^{(k)} + q_2^{(k)})$$
so that there exists $p^{(k+1)} \in F(u^{(k+1)})$ satisfying
		\begin{align}
		  0 &= p^{(k+1)} + \eta^{(k)} (u^{(k+1)} - v^{(k)} + q_1^{(k)} + u^{(k+1)} - w^{(k)} + q_2^{(k)})\\
		    &= p^{(k+1)} + \eta^{(k)} (u^{(k)} - v^{(k)} + q_1^{(k)} + u^{(k)} - w^{(k)} + q_2^{(k)}) + 2 \eta^{(k)} (u^{(k+1)} - u^{(k)}) \\
		    &= p^{(k+1)} + \eta^{(k)} \epsilon(k) + 2 \eta^{(k)} (u^{(k+1)} - u^{(k)}).
		\end{align}
		Rearranging terms, taking the norm and applying the triangle inequality leads to
		\begin{align} \label{eq:proof_eq1}
		  \| u^{(k+1)} - u^{(k)}\|_2 \le \frac{\| p^{(k+1)} \|_2}{2 \eta^{(k)}} + \frac{1}{2} \| \epsilon(k)\|_2.
		\end{align}
		Since $\|x-y\| \ge \|x\| - \|y\|$ and by assumption \eqref{subdiff_prop} it follows
		\begin{align} \label{eq:proof_eq2}
		  \| u^{(k+1)}\|_2 	&\le \frac{\| p^{(k+1)} \|_2}{2 \eta^{(k)}} + \frac{1}{2} \|\epsilon(k)\|_2 + \| u^{(k)}\|_2 \\
					&\le \frac{C \| u^{(k+1)} \|_2}{2 \eta^{(k)}} + \frac{C}{2 \eta^{(k)}}+ \frac{1}{2} \| \epsilon(k)\|_2 + \| u^{(k)}\|_2 .
		\end{align}
		Since $\frac{C}{2 \eta^{(k)}} \rightarrow 0$ as $k \rightarrow \infty$, there exists a $K$ such that
$1 < \frac{1}{1 - \frac{C}{2 \eta^{(k)}}} \le \tau := \sqrt{\sigma}$
for all $k >K$.
		Now \eqref{eq:proof_eq2} implies
		\begin{align}
		  \| u^{(k+1)}\|_2 \left(1 - \frac{C}{2 \eta^{(k)}}\right) 	&\le \frac{C}{2 \eta^{(k)}}+ \frac{1}{2} \|\epsilon(k)\|_2 + \| u^{(k)}\|_2
		\end{align}
		which gives for $k > K$ the estimates
		\begin{align}
		  \| u^{(k+1)}\|_2  	&\le \tau \frac{C}{2 \eta^{(k)}}+ \tau\frac{1}{2} \|\epsilon(k)\|_2 + \tau\| u^{(k)}\|_2 \\
					&\le \tau \frac{C}{2 \eta^{(k)}}+ \tau\frac{1}{2} \|\epsilon(k)\|_2 + \tau^2 \frac{C}{2 \eta^{(k-1)}}+ \tau^2\frac{1}{2} \|\epsilon(k-1)\|_2 + \tau^2\| u^{(k-1)}\|_2 \\
					&\le \tau^{k+1-K} \| u^{(K)}\|_2 + \sum_{j=1}^{k+1-K} \frac{C \tau^{j}}{2 \eta^{(k+1-j)}}+ \sum_{j=1}^{k+1-K} \frac{\tau^j}{2} \|\epsilon(k+1-j)\|_2 \\
					&\le \tau^{k+1} \big( \| u^{(K)}\|_2 + \sum_{j=1}^{k+1-K} \frac{C}{2 \eta^{(k+1-j)}}+ \sum_{j=1}^{k+1-K} \frac{1}{2} \|\epsilon(k+1-j)\|_2 \big)
                \end{align}
and by the exponential decay of $\| \epsilon(k) \|_2$ with $\eta^{(k)}$ further
\begin{align}
\| u^{(k+1)}\|_2 &\le C \tau^{k+1}.
		\end{align}
Using this relation together with \eqref{subdiff_prop} and \eqref{eta} in \eqref{eq:proof_eq1} we conclude
		\begin{align}
		  \| u^{(k+1)} - u^{(k)}\|_2 		&\le \frac{\| p^{(k+1)} \|_2}{2 \eta^{(k)}} + \frac{1}{2} \| \epsilon(k)\|_2 \\
						&\le \frac{C \| u^{(k+1)}\|_2 }{2 \eta^{(k)}}+ \frac{C }{2 \eta^{(k)}} + \frac{1}{2} \| \epsilon(k)\|_2  \\
						&\le \frac{C^2 \tau^{k+1} }{2 \eta^{(k)}}+ \frac{C }{2 \eta^{(k)}} + \frac{1}{2} \| \epsilon(k)\|_2 \\
						&\le \frac{C^2 }{2 \eta^{(0)} \sigma^{\frac{k-1}{2}}}+ \frac{C }{2 \eta^{(k)}} + 3\sqrt{\frac{2\lambda n}{\eta^{(k-2)}}}.
		\end{align}
		Thus, $\| u^{(k+1)} - u^{(k)}\|_2$ decreases exponentially. Therefore it is a Cauchy sequence and
		 $\{ u^{(k)} \}_k$ converges to some $\hat u$ as $k \rightarrow \infty$. Since $q_1^{(k)} \rightarrow 0$
		and $q_2^{(k)} \rightarrow 0$ as $k \rightarrow \infty$ we obtain by \eqref{b1} and \eqref{b2} that
		$\{ v^{(k)} \}_k$ and $\{ w^{(k)} \}_k$ also converge to $\hat u$. This finishes the proof.
	\end{proof}
The assumptions in the next theorem fit to our constrained models ($\mu = 1$), but are more general.
%
	\begin{theorem} \label{thm2}
		Let $F:\mathbb R^{dn} \rightarrow \mathbb R \cup \{ + \infty \}$ be any function which is bounded on its domain.
		Further assume that \eqref{admm_sol_u_gen} has a global minimizer.
                Then Algorithm \ref{A1} converges in the sense that
		$(u^{(k)}, v^{(k)}, w^{(k)}) \rightarrow (\hat u, \hat v, \hat w)$ as $k \rightarrow \infty$
                with $\hat u = \hat v = \hat w$ and $(q_1^{(k)}, q_2^{(k)}) \rightarrow (0,0)$ as $k \rightarrow \infty$.		
	\end{theorem}
%
	\begin{proof}
		As in the proof of Theorem \ref{thm1} we can show that \eqref{important} holds true for
$\epsilon(k) := v^{(k)}-u^{(k)}-q_1^{(k)} + w^{(k)}-u^{(k)}-q_2^{(k)}$.
		The quadratic term in \eqref{admm_sol_u_gen} can be rewritten as
		\begin{align}
		  \|u-v^{(k)}+q_1^{(k)}\|_2^2 + \|u-w^{(k)}+q_2^{(k)}\|_2^2  	
									&= 2 \langle u,u \rangle + 2 \langle u, q_1^{(k)} - v^{(k)} + q_2^{(k)} - w^{(k)} \rangle + C \\
									&= 2 \| u-u^{(k)} \|_2^2 - 2 \langle u, \epsilon(k) \rangle + C.
		\end{align}
		Thus, the first step of Algorithm \ref{A1} is equivalent to
		\begin{align}
		  u^{(k+1)} \in \argmin_{u} \Big\{ F(u) + \eta^{(k)} \|u-u^{(k)}\|_2^2 - \eta^{(k)} \langle \epsilon(k) , u  \rangle \Big\}.
		\end{align}
		This implies
		\begin{align}
		   F(u^{(k+1)}) + \eta^{(k)} \|u^{(k+1)}-u^{(k)}\|_2^2 - \eta^{(k)} \langle \epsilon(k) , u^{(k+1)}  \rangle &\le F(u^{(k)}) - \eta^{(k)} \langle \epsilon(k) , u^{(k)}  \rangle
                \end{align}
 and further
\begin{align}
		  \|u^{(k+1)}-u^{(k)}\|_2^2 &\le \frac{F(u^{(k)}) - F(u^{(k+1)})}{\eta^{(k)}} - \langle \epsilon(k) , u^{(k)} - u^{(k+1)}  \rangle.
		\end{align}
		Using the boundedness of $f$ and the Cauchy-Schwarz inequality leads to
		\begin{align}
		  \|u^{(k+1)}-u^{(k)}\|_2^2 &\le \frac{C}{\eta^{(k)}} + \| \epsilon(k) \|_2 \| u^{(k)} - u^{(k+1)}\|_2.
		\end{align}
		Since $\epsilon(k) \rightarrow 0$ as $k \rightarrow \infty$, we conclude that $\| u^{(k)} - u^{(k+1)}\|_2$ is bounded so that
				\begin{align}
		  \|u^{(k+1)}-u^{(k)}\|_2^2 				      &\le \frac{C}{\eta^{(k)}} + C \| \epsilon(k) \|_2.
		\end{align}
		Thus, $\| u^{(k)} - u^{(k+1)}\|_2$ is decreasing exponentially and $\{ u^{(k)} \}_k$ converges to some $\hat u$ as $k \rightarrow \infty$.
			\end{proof}
%
\section{Numerical Results} \label{sec:experiments}
%
In this section we present numerical results obtained
by our partitioning approaches.
The test images for the disparity and the optical flow problems were taken from
\begin{itemize}
\item
http://vision.middlebury.edu/stereo/ \cite{SP07,SS02,SS03}, and
\item
http://vision.middlebury.edu/flow/ \cite{BSLRBS11},
\end{itemize}
respectively.
All examples were executed on a computer with an Intel Core i7-870 Processor (8M Cache, 2.93 GHz)
and 8 GB physical memory, 64 Bit Linux.

We compare our direct partitioning methods \eqref{e_disp} and \eqref{e_flow} via Algorithm \ref{A1}
with a two-stage approach consisting of  i) disparity, resp.\ optical flow estimation, and
ii) partitioning of the estimated values. More precisely the two stage algorithm performs as follows:
\begin{itemize}
\item[i)]
In the first step, the disparity is estimated using the TV regularized model
\begin{align}\label{ex:disparity_tv}
	\min_{u_1 \in S_{Box}} & \Big\{  \frac{1}{2}\| A_1 u_1 - b_1\|_2^2 + \iota_{S_{Box}} (u_1) +
	\alpha_1 \| \, |\nabla u_1| \, \|_{1} \Big\}
\end{align}
with $A_1$ and $b_1$ defined by \eqref{a1} and \eqref{b1}, respectively.
Here $|\nabla u_1|$ stands for the discrete version of
$\left( \left( \frac{\partial u_1}{\partial x}(x,y)\right)^2 + \left( \frac{\partial u_1}{\partial y}(x,y)\right)^2\right)^\frac12$,
i.e., we use the isotropic (``rotationally invariant'') TV version.
Such model was proposed for the disparity estimation in \cite{CEFPP12} and can be found with e.g., shearlet
regularized $\ell_1$ norm in \cite{Fit13}.
For estimating the optical flow  we minimize
\begin{equation}\label{ex:flow_tv}
	\min_{u}  \Big\{  \frac{1}{2} \| A u - b\|_2^2 +
	\alpha_1 	\| \sqrt{|\nabla u_1|^2 + |\nabla u_2|^2} \|_1 \Big\},
\end{equation}
with $A$ and $b$ defined by \eqref{only_a} and \eqref{only_b}, respectively.
The global minimizers of the convex functionals \eqref{ex:disparity_tv} and \eqref{ex:flow_tv}
were computed via the primal-dual hybrid gradient method (PDHG) proposed in \cite{CP11,PCCB09}.
Clearly, one could use
other iterative first order (primal-dual) methods, see, e.g., \cite{CP10}.
\item [ii)]
In the second step the estimated disparity, resp.\ optical flow is partitioned by the method in \cite{SW14}
which minimizes, e.g., for the disparity the functional
\begin{align}
	\min_{u_1} \Big\{ \frac{1}{2} \|u_1 - u_{1,est}\|_2^2 + \alpha_2 (\| \nabla_1 u_1 \|_0 + \| \nabla_2 u_1 \|_0) \Big\},
\end{align}
where $ u_{1,est}$ is the disparity estimated in the first step.
For the approximation of a minimizer we use the software package Pottslab  http://pottslab.de
with default parameters.
Note that by introducing weights $w$ in the Potts prior the functional can be made
more isotropic which leads to a better ``rotation invariance'', see \cite{SW14}.
\end{itemize}

Next we comment on the direct partitioning implementation.
Our partitioning models \eqref{e_disp} and \eqref{e_flow} are based
on the knowledge of initial values $\bar u_1$ and $\bar u$ for the disparity, resp., the optical flow.
Here we use a simple block matching based algorithm, see \cite{CEFPP12}.
This method consists basically of a search within a given range.
For each pixel in the first image we compare
its surrounding block with surrounding blocks of pixels in the search range of the second image.
The chosen block size is $7 \times 7$. 
As a similarity measure we use the normalized cross correlation \cite{TLC03}.
Finally we apply a median filter to the initial guess to reduce the influence of outliers.
Since $(i - \bar u_1,j)$, resp.\  $(i,j) - \bar u(i,j)$
are the grid coordinates of the pixel in the second image corresponding to pixel $(i,j)$ in the first image,
we see that $f_2(i- \bar u_1,j)$, resp.\ $f_2((i,j) - \bar u)$ are really well defined grid functions.
As parameters in Algorithm \ref{A1} we choose $\eta^{(0)} = 0.01$ and $\sigma = 1.05$.
The algorithm is initialized with $v^{(0)} = w^{(0)}=\bar u_1$ for the disparity partitioning
and $v^{(0)} = w^{(0)}=\bar u$ for the flow partitioning; further $q_i^{(0)}$, $i=1,2$ are zero matrices.
We show the results after 100 iterations where no differences to subsequently iterated images can be seen.
\\

We start with the disparity partitioning results.
Figure \ref{fig:venus} shows the results for the image ``Venus''. The true disparity contains
horizontal and vertical structures so that our non isotropic direct approach fits fine.
It can compete with the more expansive two stage method. The main differences appear due
to the more or less isotropy of the models.
%
		\begin{figure}
		\centering
			\includegraphics[width=0.22 \textwidth]{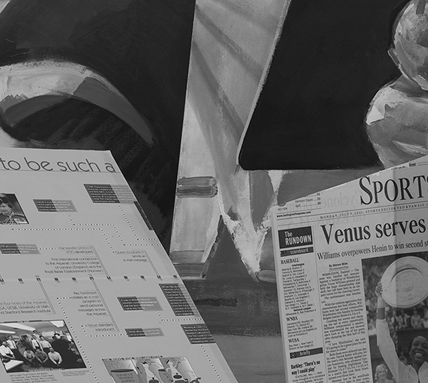}
			\hspace{0.2cm}
			\includegraphics[width=0.22 \textwidth]{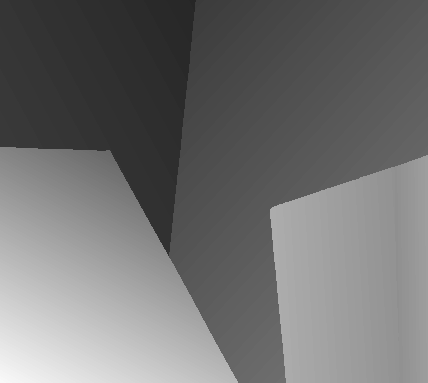}
			\hspace{0.2cm}
			\includegraphics[width=0.22 \textwidth]{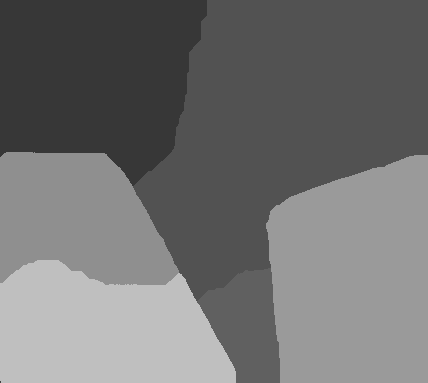}
			\hspace{0.2cm}
			\includegraphics[width=0.22 \textwidth]{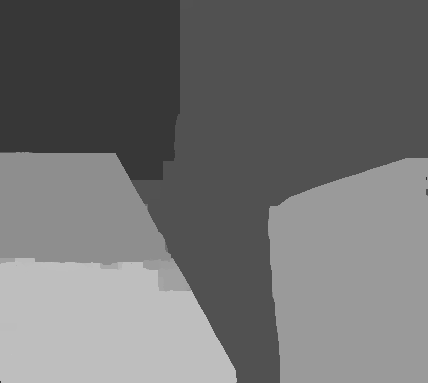}
			\caption{ \label{fig:venus}
			Results for the test images ``Venus''.
			Left to right: original left image, ground truth,
			partitioned disparity
			using the two stage algorithm  ($\alpha_1 = 0.005$, $\alpha_2 = 300$),
			partitioned	disparity using the direct algorithm ($\lambda = 2.5$).
			}
			\end{figure}
%

Figs. \ref{fig:cones} and \ref{fig:dolls} show that our direct partitioning algorithm can qualitatively
compete with the two stage algorithm.

%
		\begin{figure}[ht]
		\centering
			\includegraphics[width=0.22 \textwidth]{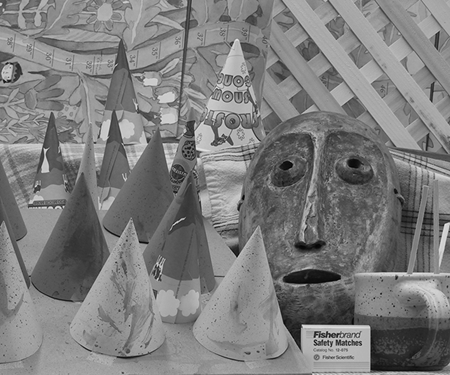}
			\hspace{0.2cm}
			\includegraphics[width=0.22 \textwidth]{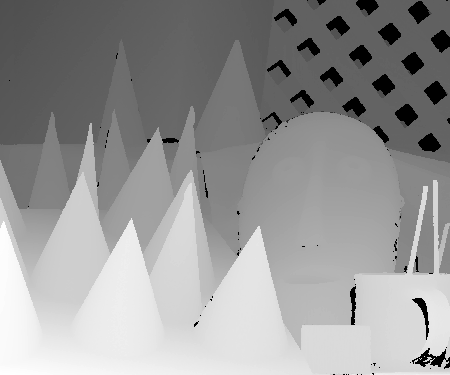}
			\hspace{0.2cm}
			\includegraphics[width=0.22 \textwidth]{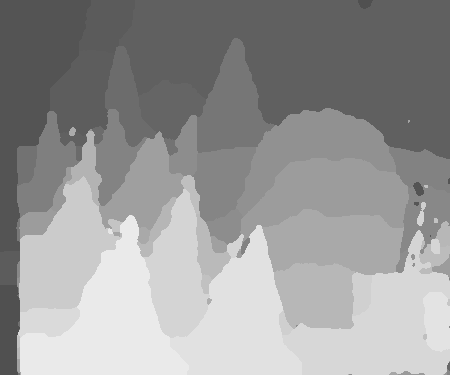}
			\hspace{0.2cm}
			\includegraphics[width=0.22 \textwidth]{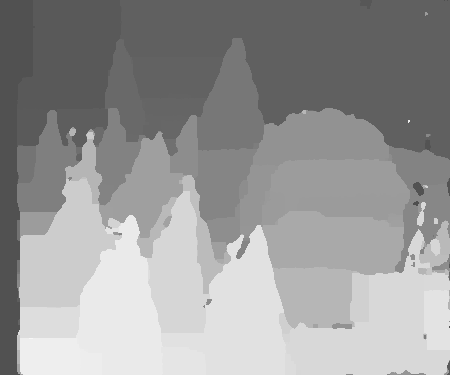}
			\caption{ \label{fig:cones}
			Result for the images ``Cones''. Left to right: original left image, ground truth,
			partitioned disparity using the two stage algorithm ($\alpha_1 = 0.005$, $\alpha_2 = 50$),
			partitioned	disparity using the direct algorithm ($\lambda = 0.5$).}
					\end{figure}
	%
	
		\begin{figure}[ht]
		\centering
			\includegraphics[width=0.22 \textwidth]{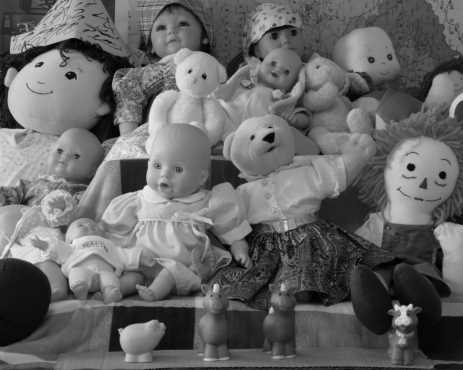}
			\hspace{0.2cm}
			\includegraphics[width=0.22 \textwidth]{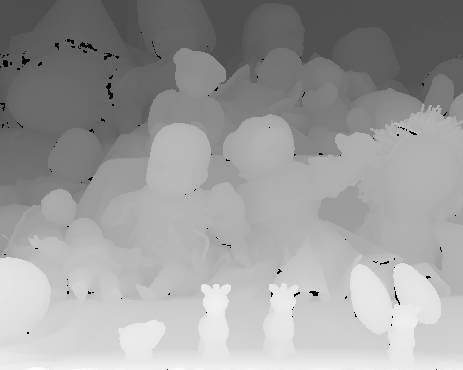}
			\vspace{0.2cm}
			\includegraphics[width=0.22 \textwidth]{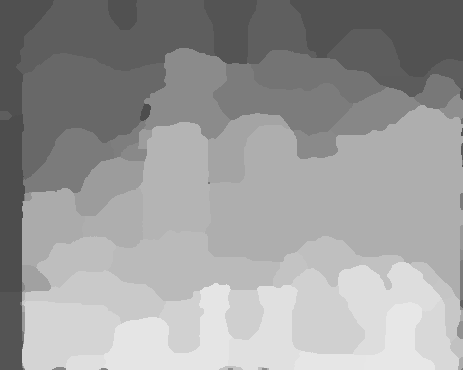}
			\hspace{0.2cm}
			\includegraphics[width=0.22 \textwidth]{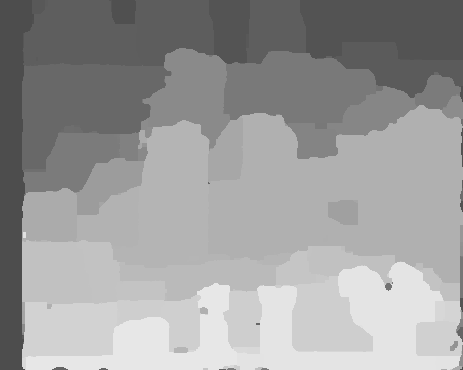}
			\caption{\label{fig:dolls}
			Result for the ``Dolls'' images.
			Left to right: original left image, ground truth,
			partitioned disparity using the two stage algorithm ($\alpha_1 = 0.01$, $\alpha_2 = 80$),
			partitioned	disparity using the direct algorithm ($\lambda = 1.5$).}
			\end{figure}
		%
		
Next we show our results for the optical flow partitioning.
The flow vectors are color coded with
color $\simeq$ direction, brightness $\simeq$ magnitude).
The ground truth flow field in the first example ``Wooden'' in Fig. \ref{fig:wooden}
prefers horizontal and vertical directions. As in the first disparity example our algorithm show good results.
In Fig. \ref{fig:rubber} and Fig. \ref{fig:hydra} we see that our direct method can compete
with the more involved two stage approach.
The main differences appear again due to the more isotropic approach in the two stage model.
Especially in Fig. \ref{fig:rubber} one can see that the flow field of the rotating wheel is partitioned into 
rectangular instead of annular segments by our direct method. In the same figure, we show a result where
we have estimated the optical flow in Step 1 by the more sophisticated model in \cite{BM11}, 
for the program code see  http://lmb.informatik.uni-freiburg.de/resources/software.php. Step 2 was the same.
The result is only slightly different from those obtained by the previously described two stage algorithm.

%
		\begin{figure}[ht]
		\centering
			\includegraphics[width=0.22 \textwidth]{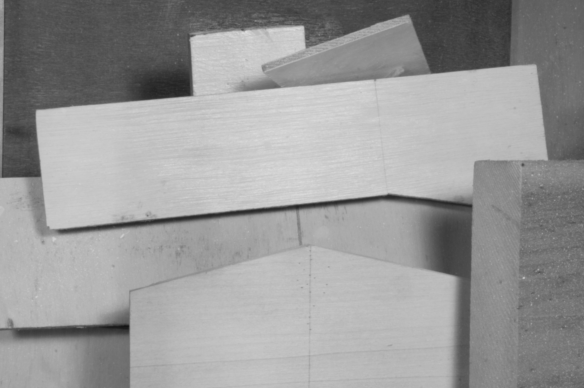}
			\hspace{0.2cm}
			\includegraphics[width=0.22 \textwidth]{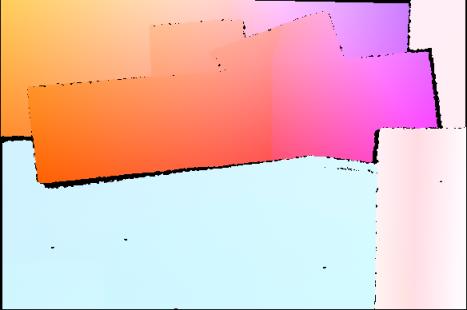}
			\hspace{0.2cm}
			\includegraphics[width=0.22 \textwidth]{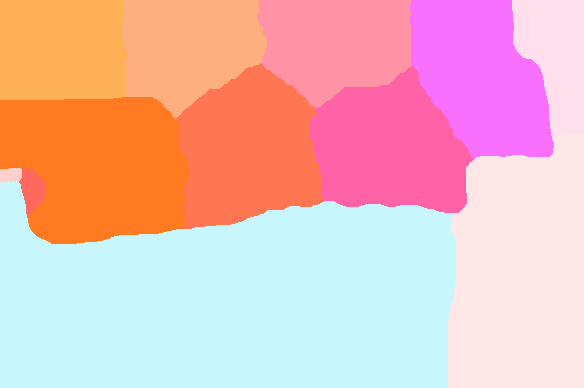}
			\hspace{0.2cm}
			\includegraphics[width=0.22 \textwidth]{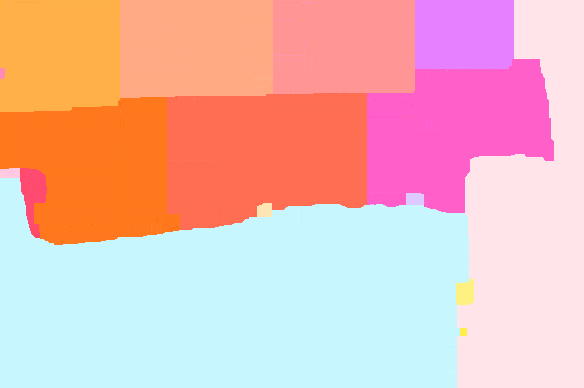}
			\caption{\label{fig:wooden}
			Result for the ``Wooden'' images, Left to right:
			first test image, ground truth, partitioned optical flow using the two stage algorithm ($\alpha_1 = 0.01$, $\alpha_2 = 150$),   
partitioned optical flow by the direct algorithm ($\lambda = 0.5$).
			}
				\end{figure}
%
		\begin{figure} \centering
			\includegraphics[width=0.3 \textwidth]{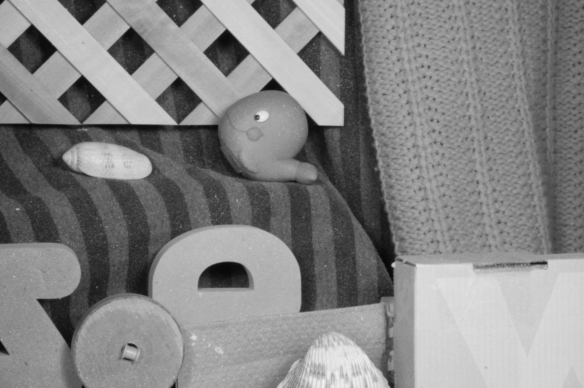}
			\hspace{0.2cm}
			\includegraphics[width=0.3 \textwidth]{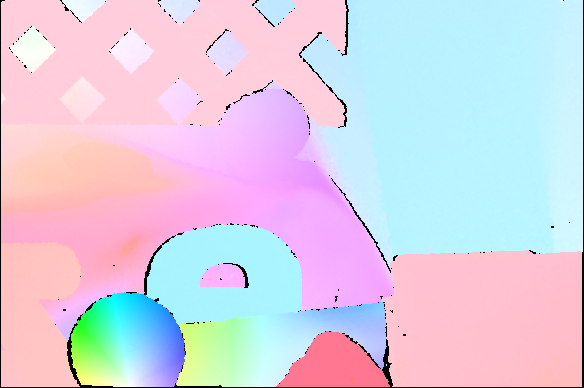}
			\hspace{0.2cm}
			\includegraphics[width=0.3 \textwidth]{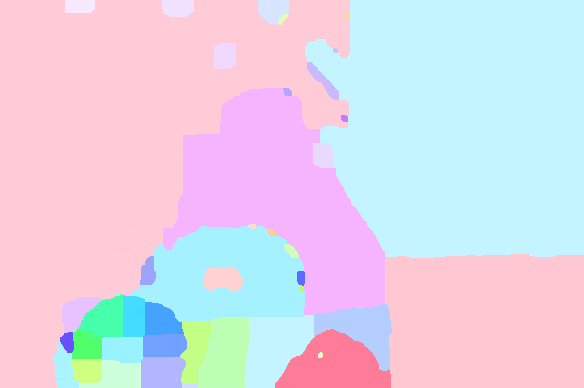}
			\\[2ex]			
			\includegraphics[width=0.3 \textwidth]{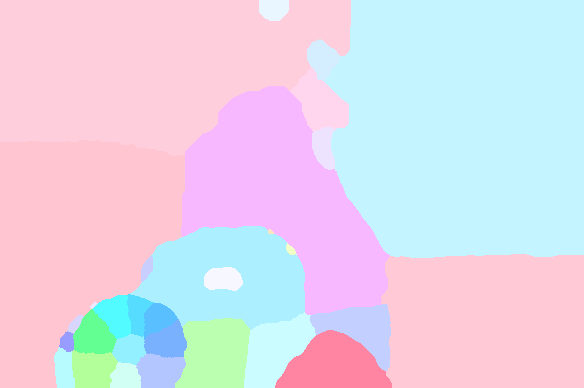}
			\hspace{0.2cm}
			\includegraphics[width=0.3 \textwidth]{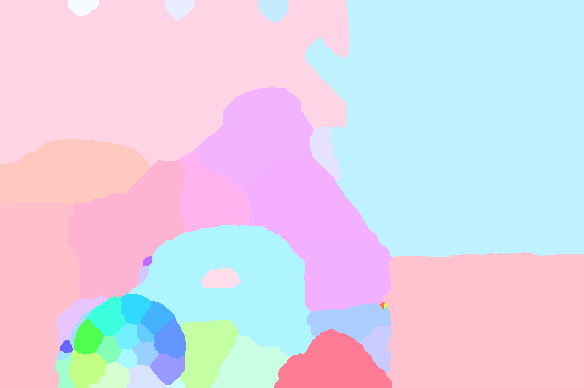}
			\caption{
                        \label{fig:rubber}
			Result for the test images ``RubberWhale''.
			Top: first test image,  ground truth, partitioned optical flow by the direct algorithm ($\lambda = 0.05$) .
			Bottom: partitioned optical flow by the two stage algorithm. 
                        Left: Two stage algorithm ($\alpha_1 = 0.005$, $\alpha_2 = 7$), 
                        Right: Two stage algorithm but with Step 1 computed by the model in \cite{BM11} ($\alpha_2 = 7$). }
			\end{figure}
		
		\begin{figure} \centering
			\includegraphics[width=0.22 \textwidth]{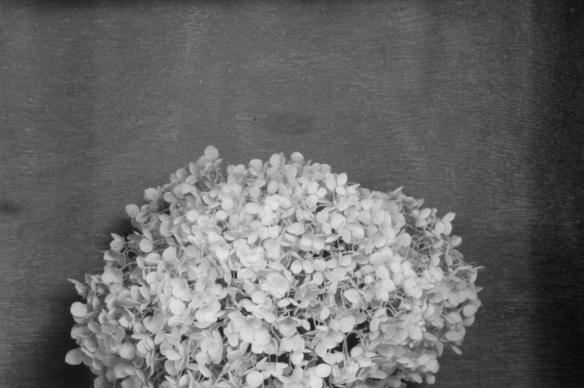}
			\hspace{0.2cm}
			\includegraphics[width=0.22 \textwidth]{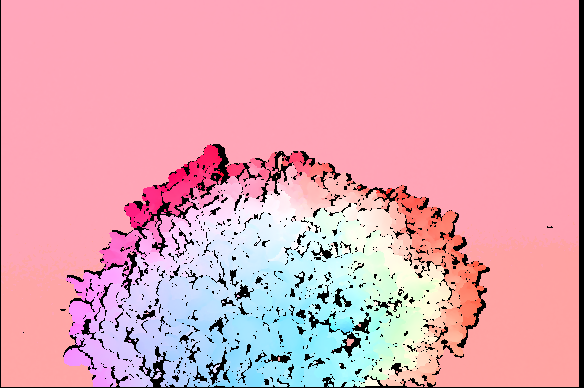}
			\hspace{0.2cm}
			\includegraphics[width=0.22 \textwidth]{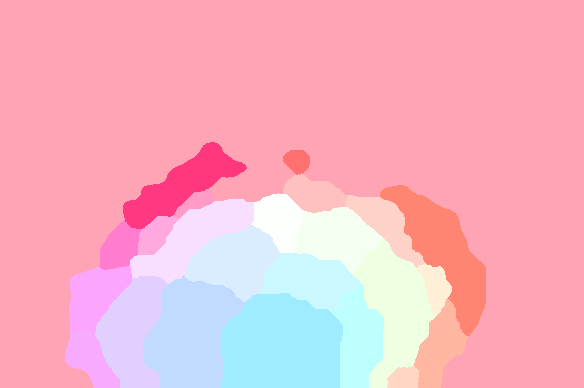}
			\hspace{0.2cm}
			\includegraphics[width=0.22 \textwidth]{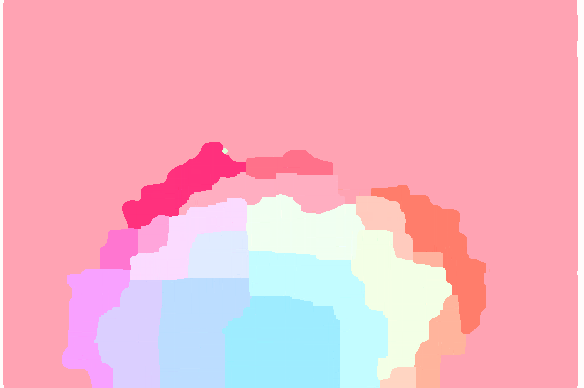}
			\caption{\label{fig:hydra}
			Result for the images ``Hydrangea''.
			Left to right:  ground truth, partitioned optical flow by the two stage algorithm
			($\alpha_1 = 0.01$, $\alpha_2 = 35$),  partitioned optical flow by the direct algorithm ($\lambda = 0.15$).}
				\end{figure}

\section{Conclusions} \label{sec:conclusions}
%
In this paper we have proposed a new method
for disparity and optical flow partitioning based on a Potts regularized
variational model together with an ADMM like algorithm.
In case of the optical flow it is adapted to vector-valued data.
In this paper, we have only shown the {\it basic approach} and further refinements are planned in the future.
So we intend to incorporate more sophisticated data fidelity terms.
In particular illumination changes should be handled.
We will make the model more ``rotationally invariant''.
The simple introduction of weights and other differences as in \cite{SW14} and in several graph cut approaches
is one possibility.
The crucial part for the running time of the proposed direct algorithm is the univariate Potts minimization.
However, since the single problems are independent of each other, they could be solved in parallel.
Such parallel implementation is another point of future activities.
Further we want to incorporate multiple frames instead of just two of them in our model.
From the theoretical point of view,
to establish just the convergence of an algorithm to a local minimizer
seems not to be enlightening since certain constant images
are contained in the set of local minimizers and we are clearly not looking for them.
However, a better understanding of strict (local) minimizers and the choice of initial values for the algorithm
is interesting.
\\

%
{\bf Acknowledgement:}
The work of J. H. Fitschen has been supported by Deutsche Forschungsgemeinschaft (DFG) within the Graduate School 1932.
Some parts of the paper have been written during a visit M. Nikolova at this Graduate School.
Many thanks to M. El-Gheche (University Paris Est) for fruitful discussions on disparity estimation.
%


\bibliographystyle{abbrv}


\end{document}